\documentclass[11pt]{article}
\usepackage{amssymb,epsfig}
\usepackage{amsmath,amsthm}
\usepackage{synttree}

\newtheorem{theorem}{Theorem}[section]
\newtheorem{prop}[theorem]{Proposition}
\newtheorem{lemma}[theorem]{Lemma}

\newcommand\beq{\begin{equation}}
\newcommand\eeq{\end{equation}}
\newcommand\bce{\begin{center}}
\newcommand\ece{\end{center}}
\newcommand\bea{\begin{eqnarray}}
\newcommand\eea{\end{eqnarray}}
\newcommand\bean{\begin{eqnarray*}}
\newcommand\eean{\end{eqnarray*}}
\newcommand\ben{\begin{enumerate}}
\newcommand\een{\end{enumerate}}
\newcommand\bit{\begin{itemize}}
\newcommand\eit{\end{itemize}}
\newcommand\brr{\begin{array}}
\newcommand\err{\end{array}}
\newcommand\bt{\begin{tabular}}
\newcommand\et{\end{tabular}}

\newcommand\nn{\nonumber}

\newcommand\ms{\medskip}

\renewcommand\S{\mathfrak S}
\newcommand\D{\mathcal D}

\newcommand\M{\mathcal M}
\newcommand\wh{\widehat}
\newcommand\wt{\widetilde}
\newcommand\RS{\mathcal {RS}}
\newcommand\bij{\varphi}

\newcommand\bijr{\psi}
\newcommand\red{\rho}
\newcommand\krar{\Phi}
\newcommand\T{\mathcal T}
\newcommand\udt{\dot{\,}}
\newcommand\A{\mathcal A}

\author{Emeric Deutsch~\thanks{Polytechnic Institute of New York University, Brooklyn, NY 11201.}
\and
Sergi Elizalde~\thanks{Department of Mathematics,
Dartmouth College, Hanover, NH 03755.}} 
\title{Restricted simsun permutations}
\date{}

\begin{document}
\maketitle

\hfill {\it To the memory of Rodica Simion}\ms

\begin{abstract}
A permutation is simsun if for all $k$, the subword of the one-line notation consisting of the $k$ smallest entries does not have three consecutive decreasing elements.
Simsun permutations were introduced by Simion and Sundaram, who showed that they are counted by the Euler numbers.
In this paper we enumerate simsun permutations avoiding a pattern or a set of patterns of length $3$.
The results involve Motkzin, Fibonacci, and secondary structure numbers. The techniques in the proofs include generating functions, bijections into lattice paths and generating trees.
\end{abstract}

\section{Introduction}

\subsection{Simsun permutations}

A permutation $\pi=\pi_1\pi_2\dots\pi_n\in\S_n$ is called {\em simsun} if for all $3\le k\le n$, the restriction of $\pi$ to $\{1,2,\dots,k\}$ has no double descents, that is, there are no three consecutive entries in
decreasing order. For example, $41325$ is simsun and $32415$ is not. Simsun permutations are named after Rodica
{\it Sim}ion and Sheila {\it Sun}daram~\cite{Sun}. We denote by $\RS_n$ the set of simsun permutations in $\S_n$. Simion and Sundaram proved that simsun permutations are enumerated by the Euler numbers.
To state their result, we denote by $E_n$ the $n$-th Euler number, which is known to count permutations $\pi\in\S_n$ with $\pi_1<\pi_2>\pi_3<\pi_4>\dots$, usually called up-down permutations.

\begin{theorem}\label{th:rs} $$|\RS_n|=E_{n+1}.$$
\end{theorem}

The proof of the above result uses an important property of simsun permutations: deleting $n$ from $\pi\in\RS_n$ results in a permutation in $\RS_{n-1}$. Note that the analogous property for
up-down permutations does not hold.

This fact can be used to build a generating tree $\T$ for simsun permutations as follows.
For each permutation in $\RS_n$, its {\em fertility} positions are those spots between adjacent entries or at the beginning or at the end of the permutation
where by inserting $n+1$ we obtain a permutation in $\RS_{n+1}$.
The nodes at level $n$ of the tree are the simsun permutations of length $n$, and the children of a node $\pi\in\RS_n$ are the permutations obtained by inserting
$n+1$ in a fertility position of $\pi$. The root of the tree is the permutation $1$. The above property implies
that this tree contains all simsun permutations.
It is clear from the definition that the fertility positions of $\pi\in\RS_n$ are all the spots except those immediately preceding a descent of $\pi$.

Simsun permutations, which are a variant of {\em Andr\'e permutations}~\cite{FS},
have an interesting connection with the $cd$-index of $\S_n$. The coefficients of the monomials of the $cd$-index of $\S_n$
are equal to the number of simsun permutations in $\S_{n-1}$ with a given descent set (see~\cite{Het,HR} for details).
More recently, Chow and Shiu \cite{CS} have enumerated simsun permutations with respect to the number of descents.

In this paper we study simsun permutations that avoid patterns of length 3. We say that $\pi\in\S_n$ {\it avoids} $\sigma\in\S_k$ if there are no indices $i_1<i_2<\dots<i_k$ such that
$\red(\pi_{i_1}\pi_{i_2}\dots\pi_{i_k})=\sigma$, where $\red$ is the reduction consisting of relabeling the smallest element with 1, the second smallest with 2, and so on.
We denote by $\S_n(\sigma)$ the set of $\sigma$-avoiding permutations in $\S_n$, and by $\RS_n(\sigma)$ the set of $\sigma$-avoiding permutations in $\RS_n$.
We define $\S_n(\sigma,\tau,\dots)$ and $\RS_n(\sigma,\tau,\dots)$ similarly by requiring the permutations to avoid several patterns at the same time.
In~\cite{SS}, Simion and Schmidt enumerated permutations avoiding any subset of the set of six patterns of length~3.

A more general version of pattern avoidance, introduced in~\cite{BS}, allows the requirement that some entries have to be adjacent in an occurrence of the pattern
in a permutation. In order to be consistent with the above notation for classical pattern avoidance, we will indicate the positions of the pattern
that are required to be adjacent by putting a hat over them. For example, $\pi$ avoids the generalized pattern $\wh{53}4\wh{162}$ if there are no
indices $i+1<j<k$ such that $\pi_k<\pi_{k+2}<\pi_{i+1}<\pi_j<\pi_i<\pi_{k+1}$.

An equivalent formulation of the simsun condition is that there do not exist indices $a<b<c$ with $\pi_a>\pi_b>\pi_c$ and so that
$\pi_i>\pi_a$ for all $a<i<b$ and all $b<i<c$. Consequently, we have that \beq\label{eq:inclRS}\S_n(\wh{32}1)\subseteq\RS_n\subseteq\S_n(\wh{321}).\eeq Indeed, if such indices exist, then $\pi_{b-1}\pi_b\pi_c$ is an occurrence of
$\wh{32}1$, which proves the first inclusion. The second inclusion is clear since an occurrence of $\wh{321}$ is a double descent, which is not allowed in a simsun permutation. Both inclusions are strict when $n\ge4$.

\subsection{Sequences and lattice paths}\label{sec:background}

In Table~\ref{tab:seq} we set the notation for and give the definition of some
of the sequences of integers that will appear in the paper.
\begin{table}[htb]
\begin{tabular}{c|c|c|c}
Notation & Name & Generating function & OEIS~\cite{OEIS}\\
\hline
$C_n$ & Catalan & $\sum_{n\ge0} C_n z^n= \frac{1-\sqrt{1-4z}}{2z}$ &  A000108\\
$M_n$ & Motkzin & $\sum_{n\ge0} M_n z^n = \frac{1-z-\sqrt{1-2z-3z^2}}{2z^2}$ & A001006 \\
$S_n$ & Secondary structure & $\sum_{n\ge0} S_n z^n = \frac{1-z-z^2-\sqrt{1-2z-z^2-2z^3+z^4}}{2z^3}$  & A004148\\
$F_n$ & Fibonacci & $\sum_{n\ge0} F_n z^n =\frac{z}{1-z-z^2}$ & A000045\\
$E_n$ & Euler & $\sum_{n\ge0} E_n\frac{z^n}{n!}=\sec z+\tan z$ & A000111
\end{tabular}
\caption{\label{tab:seq} Some important sequences.}
\end{table}

Recall that a Dyck paths of semilength $n$ is a lattice path from $(0,0)$ to $(2n,0)$ with steps $U=(1,1)$ and $D=(1,-1)$ that never goes below the $x$-axis. The set of Dyck paths of semilength $n$ is denoted by $\D_n$, and
its cardinality is $C_n$.
A Motzkin path of length $n$ is a lattice paths from $(0,0)$ to $(n,0)$ with steps $U=(1,1)$, $D=(1,-1)$ and $H=(1,0)$ that never goes below the $x$-axis. The set of Motzkin paths of length $n$ is denoted $\M_n$,
and its cardinality is $M_n$.

We use $S_n$ to denote the number of secondary structures with $n$ vertices, which can be defined as
noncrossing matchings with no arcs between adjacent vertices (a slightly different definition is given
in~\cite[p.~241]{EC2}).
The name for these structures, which are in simple bijection with peakless Motzkin paths,
is due to the fact that they model secondary structures of RNA molecules~\cite{SW}.
The secondary structure numbers $S_n$
are also known to count the number of Dyck paths of semilength $n$ with no $UUU$ and no $DDD$, and also the
number of Dyck paths of semilength $n+2$ with no $UDU$ and no $DUD$, as shown in~\cite{Don} and in Subsection~\ref{sec:Dyck} below.
A {\em peak} in a Dyck path is an occurrence of $UD$. We will use the term {\em ascent} (resp. {\em descent}) to refer to a maximal consecutive substring of $U$ (resp. $D$) steps.

When we deal with sequences of sets indexed by $n$, such as $\D_n$ or $\S_n$, we will omit the subscript $n$ to denote the union of all the sets in the sequence. For example, $\M=\bigcup_{n\ge0}\M_n$ and
$\RS(\sigma)=\bigcup_{n\ge0}\RS_n(\sigma)$.

\subsection{Bijection to Dyck paths}\label{sec:krar}

We will use a well-known bijection between $132$-avoiding permutations and Dyck paths due to Krattenthaler~\cite{Kra}.
The following graphical description, up to rotation of the diagram by $90$ degrees, is taken from~\cite{EliPak}.
Any permutation $\pi\in\S_n$ can be represented as an $n\times n$ array
with crosses in positions $(i,\pi_i)$, where the first coordinate is the column (increasing from left to right) and the second coordinate is the row (increasing from bottom to top).
Given the array of $\pi\in\S_n(132)$, consider the path from the upper-left corner to the lower-right corner
with steps south and east that leaves all the crosses to its right and stays always as close to the diagonal connecting these two corners as possible.
Replacing each south step with a $U$ and each east step with a $D$ produces a Dyck path $\krar_{132}(\pi)\in\D_n$.

The map $\krar_{132}$ is a bijection, and its inverse map can be described as follows.
Given a Dyck path of semilength $n$, replace each $U$ with a south step and each $D$ with an east step to obtain a path from the
upper-left corner to the lower-right corner of an $n\times n$ array.
To recover the permutation, in each row from $1$ to $n$ put a
cross as far to the left as possible with the conditions that it stays to the right of the path and that no other cross has been placed in that column.
Note that the peaks of $\krar_{132}(\pi)$ occur at the left-to-right minima of $\pi$.

One can define similar bijections $\krar_{231}$, $\krar_{312}$ and $\krar_{213}$ from $\S_n(231)$, $\S_n(312)$ and $\S_n(213)$ to $\D_n$, respectively, by considering paths between the
appropriate corners of the array, as shown in Figure~\ref{fig:krar}.

\begin{figure}[hbt]
\centering
\epsfig{file=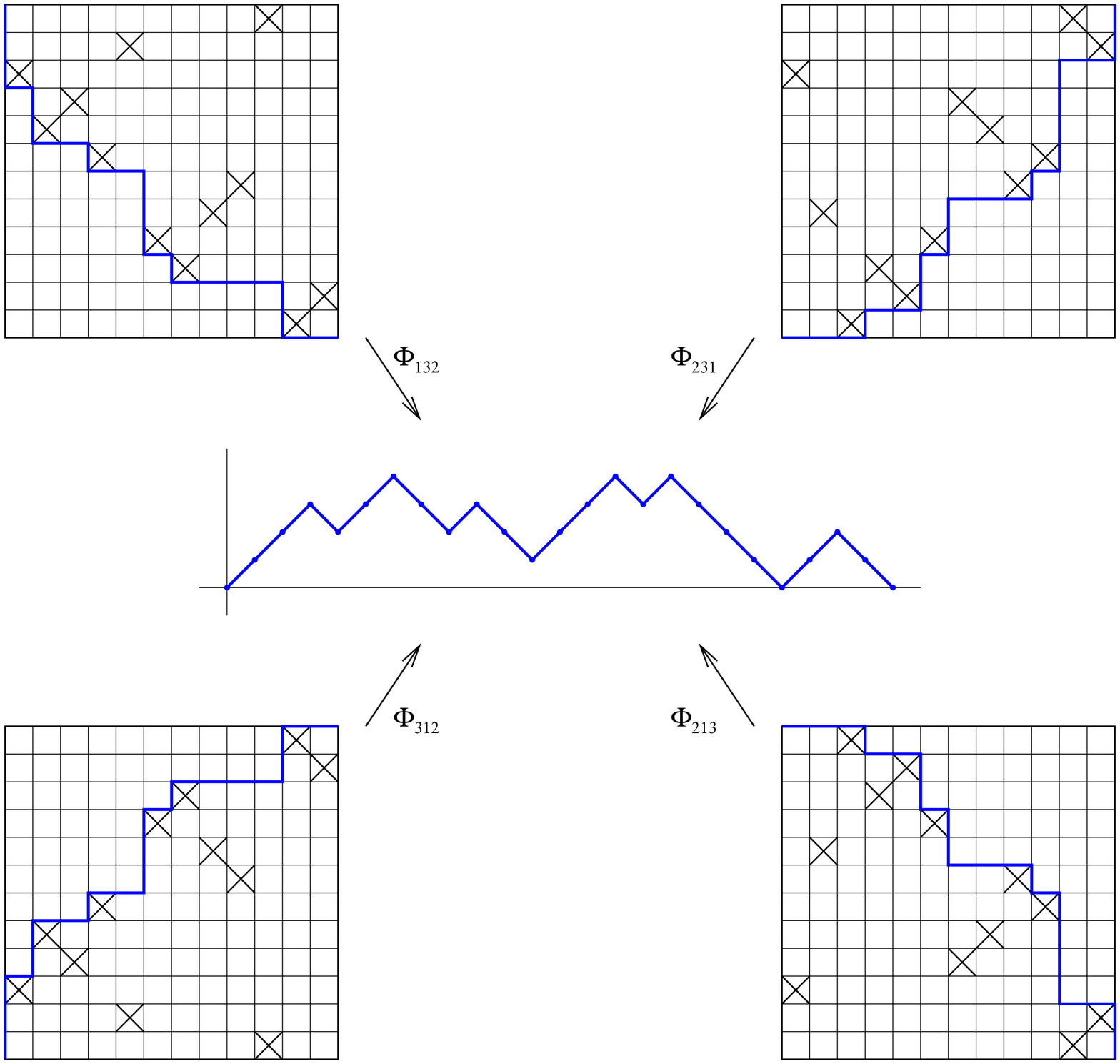,width=5in}
\caption{\label{fig:krar} The bijections $\krar_{132}(10\,8\,9\,7\,11\,4\,3\,5\,6\,12\,1\,2)$, $\krar_{231}(10\,5\,1\,3\,2\,4\,9\,8\,6\,7\,12\,11)$,
$\krar_{312}(3\,5\,4\,6\,2\,9\,10\,8\,7\,1\,12\,11)$ and $\krar_{213}(3\,8\,12\,10\,11\,9\,4\,5\,7\,6\,1\,2)$.}
\end{figure}

It is a commonly used fact that a permutation $\pi\in\S_n(132)$ with $n\ge1$ can be uniquely decomposed as $\pi=\sigma n\tau$, where all the entries in $\sigma$ are larger than all the entries in $\tau$
(we denote this by $\sigma\gg\tau$),
$\sigma\in\S(132)$, and $\red(\tau)\in\S(132)$. This yields the equation $C(z)=1+zC(z)^2$ for the generating function $C(z)$ of $132$-avoiding permutations, from where $C(z)=\frac{1-\sqrt{1-4z}}{2z}$.
We will use a similar decomposition for simsun permutations avoiding some patterns of length~3.

\subsection{Structure of the paper}

In the next five sections we enumerate simsun permutations avoiding each one of the patterns $123$, $132$, $213$, $231$ and $312$.
The case of $321$-avoiding simsun permutations is trivial, since $\RS_n(321)=\S_n(321)$ and it is well known~\cite{SS} that $|\S_n(321)|=C_n$.
In Section~\ref{sec:double} we enumerate simsun permutations avoiding each pair of patterns of length $3$. Table~\ref{tab:1or2} summarizes our results on the enumeration of simsun permutations with one and two restrictions.
To our knowledge, the case of $132$-avoiding simsun permutations is the first occurrence of the secondary structure numbers in the enumeration of permutations.

\begin{table}[htb]
$$\begin{array}{|c|c|}
\hline  \sigma & |\RS_n(\sigma)| \\ \hline
  123  & 6 \mbox{ \,(for }n\ge4) \\
  132  & S_n\\
  213  & M_n\\
  231  & M_n\\
  312  & 2^{n-1}\\
  321  & C_n \\ \hline
\end{array}\hspace{15mm}
\begin{array}{|c|c|}
\hline  \{\sigma,\tau\} & |\RS_n(\sigma,\tau)| \\ \hline
  \{123,132\}  & 2 \mbox{ \,(for }n\ge4)\\
  \{123,213\}  & 3 \mbox{ \,(for }n\ge3)\\
  \{123,231\},\{123,312\},\{123,321\}  & 0 \mbox{ \,(for }n\ge5)\\
  \{132,213\},\{213,231\},\{231,312\}  & F_{n+1}\\
  \{132,231\},\{132,312\},\{213,312\} & n\\
  \{132,321\},\{213,321\}  & \frac{1}{2}(n^2-n+2)\\
  \{231,321\},\{312,321\}  & 2^{n-1} \\ \hline
\end{array}
$$
\caption{\label{tab:1or2} The number of simsun permutations avoiding one or two patterns of length 3.}
\end{table}

For simsun permutations avoiding three or more patterns, the enumeration is
either trivial or follows easily from the results given here and in~\cite{SS}.
For completeness, we include the results in Table~\ref{tab:3up}.

\begin{table}[htb]
$$\begin{array}{|c|c|}
\hline  \Sigma & |\RS_n(\Sigma)| \\ \hline
  \{123,132,213\},\{132,213,231\},\{132,213,312\},\{132,231,312\},\{213,231,312\}  & 2 \\
  \{132,213,321\},\{132,231,321\},\{132,312,321\},\{213,231,321\},\{213,312,321\}  & n \\
  \{231,312,321\}  & F_{n+1}\\
  \mbox{all other sets of three patterns} & 0 \\ \hline
  \{132,213,231,312\} & 1 \\
  \{132,213,231,321\},\{132,213,312,321\},\{132,231,312,321\},\{213,231,312,321\}  & 2 \\
  \mbox{all other sets of four patterns} & 0 \\ \hline
  \{132,213,231,312,321\} & 1 \\
  \mbox{all other sets of five patterns} & 0 \\ \hline
\end{array}
$$
\caption{\label{tab:3up} The number of simsun permutations avoiding three or more patterns of length~3.}
\end{table}

\section{Avoiding $123$}

This is the easiest case of simsun permutations avoiding one pattern of length 3. Consider the subtree of $\T$ consisting of $123$-avoiding simsun permutations. The first four levels of this tree not including the
empty permutation are drawn in Figure~\ref{fig:RS123}.
For the remaining levels we have the following result.

\begin{figure}[htb]
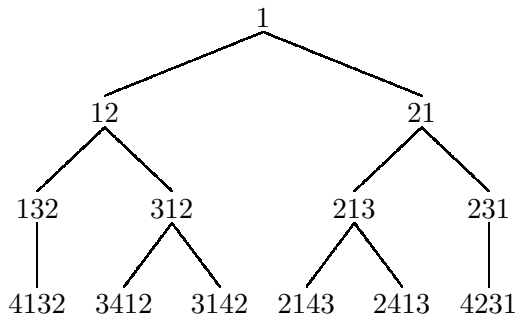

\begin{center}\synttree
[$1$ [$12$ [ $132$ [$4132$]] [$312$ [$3412$] [$3142$]]] [$21$ [$213$ [$2143$][$2413$]] [$231$ [$4231$]]]]
\end{center}
\caption{The first four levels of the generating tree for $\RS(123)$. \label{fig:RS123}}
\end{figure}

\begin{prop}\label{prop:RS123}
For $n\ge4$, $|\RS_n(123)|=6$.
\end{prop}

\begin{proof}
We claim that for $n\ge4$, each permutation $\pi\in\RS_n(123)$ has a unique fertility position where the insertion of $n+1$ produces a permutation in $\RS_{n+1}(123)$.
Indeed, we know that such a position cannot immediately precede a descent, so it must either precede an ascent or be one the two rightmost spots. On the other hand, such a fertility position
cannot have an ascent of $\pi$ to its left, otherwise the insertion of $n+1$ would create a $123$ pattern. Since all simsun permutations of length $n\ge3$ have an ascent within the first $3$ entries, the only fertility
position of $\pi\in\RS_n(123)$ followed by an ascent and with no ascents to its left is the position immediately preceding the first ascent of $\pi$.
\end{proof}

\section{Avoiding $132$}

Let $R(z)=\sum_{n\ge0} |\RS_n(132)| z^n$. Simsun permutations avoiding $132$ are counted by the secondary structure numbers.
We give three different proofs of this fact.

\begin{theorem}
$$|\RS_n(132)|=S_n.$$
\end{theorem}

\begin{proof}[First proof (generating tree)]
Consider the subtree $\T_{132}$ of $\T$ consisting of $132$-avoiding simsun permutations. In the rest of this proof, {\it fertility positions} of $\pi\in\RS_n(132)$
will refer only to those where the insertion of $n+1$
produces a permutation in $\RS_{n+1}(132)$. Note that the fertility positions are those that do not precede a descent and such that all the entries to their left are larger than all the entries to their right.
To each $\pi\in\RS_n(132)$, assign the following label to the corresponding node in $\T_{132}$:
\bit\item $(k)$ \ if $\pi$ has $k$ fertility positions and either $\pi_1<\pi_2$ or $n=1$;
\item $\wt{(k)}$  \ if $\pi$ has $k$ fertility positions and $\pi_1=n>1$;
\item $\wh{(k)}$  \ if $\pi$ has $k$ fertility positions and $\pi_2<\pi_1<n$.
\eit
Then the labels of $\T_{132}$ obey the following {\em colored} succession rules (using the terminology from~\cite{FPPR}):
\bean (k) &\longrightarrow& \wt{(k)}(k)(k-1)\dots(2),\\ \wt{(k)}&\longrightarrow&(k+1)\wh{(k-1)}\wh{(k-2)}\dots\wh{(1)},\\ \wh{(k)}&\longrightarrow&\wh{(k)}\wh{(k-1)}\dots\wh{(1)},\eean
with $(2)$ being the root, corresponding to the permutation of length 1.
The succession rules are obtained using the following facts:\ben\renewcommand{\labelenumi}{(\roman{enumi})}
\item after an insertion, the fertile positions to the right of where the insertion took place remain fertile, including the spot right next to the inserted entry;
\item the leftmost position is only fertile when $\pi_1<\pi_2$ or $n=1$, and if the insertion takes places there, then the new leftmost position is no longer fertile;
\item after an insertion in any position other than the leftmost, all the positions to the left of where the insertion took place become infertile, except the leftmost position, which remains fertile if it already was,
and becomes fertile if the insertion takes place between $\pi_1=n$ and $\pi_2$.
\een
The first four levels of $\T_{132}$ and their labels are drawn in Figure~\ref{fig:gentree132}.

\begin{figure}[htb]
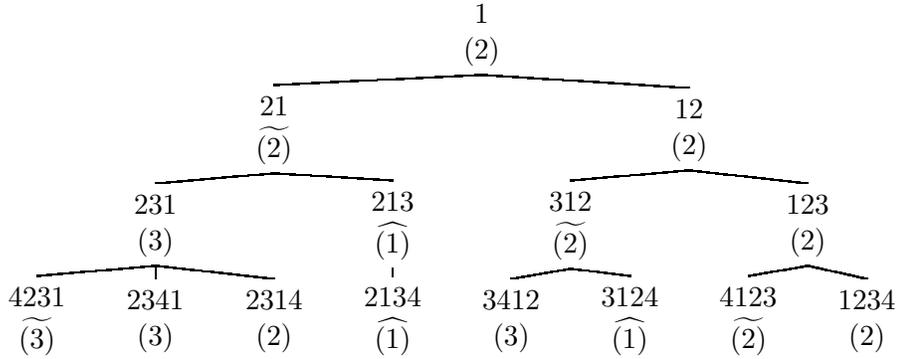

\begin{center}\synttree
[\bt{c}$1$\\$(2)$\et  [\bt{c}$21$\\ $\wt{(2)}$\et [\bt{c}$231$\\ $(3)$\et [\bt{c}$4231$\\ $\wt{(3)}$\et] [\bt{c}$2341$\\ $(3)$\et] [\bt{c}$2314$\\ $(2)$\et]] [\bt{c}$213$\\ $\wh{(1)}$\et [\bt{c}$2134$\\ $\wh{(1)}$\et]]] [\bt{c}$12$\\ $(2)$\et [\bt{c}$312$\\ $\wt{(2)}$\et [\bt{c}$3412$\\ $(3)$\et][\bt{c}$3124$\\ $\wh{(1)}$\et]] [\bt{c}$123$\\ $(2)$\et [\bt{c}$4123$\\ $\wt{(2)}$\et][\bt{c}$1234$\\ $(2)$\et]]]]
\end{center}
\caption{The first four levels of the generating tree $\T_{132}$ with the corresponding labels. \label{fig:gentree132}}
\end{figure}

Now let $A(u,z)$, $\wt{A}(u,z)$ and $\wh{A}(u,z)$ be the generating functions where the coefficient of $u^kz^n$ is the number of nodes at level $n$ (the root being at level $1$) with label $(k)$, $\wt{(k)}$ and $\wh{(k)}$, respectively.
Then the succession rules above translate into the system of equations
\bean A(u,z)&=&u^2z+\frac{uz}{u-1}(A(u,z)-uA(1,z))+uz\wt{A}(u,z), \\ \wt{A}(u,z)&=&zA(u,z), \\ \wh{A}(u,z)&=&\frac{z}{u-1}(\wt{A}(u,z)-u\wt{A}(1,z))+\frac{uz}{u-1}(\wh{A}(u,z)-\wh{A}(1,z)).\eean
These functional equations can be solved by routine use of the kernel method
(see, for example, \cite{BBDFG}).
Indeed, substituting the second one into the first, the kernel multiplying $A(u,z)$ is $$1-\frac{uz}{u-1}-uz^2,$$ which is canceled setting
$$u=\frac{1-z+z^2-\sqrt{1-2z-z^2-2z^3+z^4}}{2z^2};$$
then we can solve first for $A(1,z)$ and then for $A(u,z)$. Substituting the resulting expression for $\wt{A}(u,z)$ into the third equation and canceling the kernel $1-uz/(u-1)$ by setting $u=1/(1-z)$, we obtain $\wh{A}(1,z)$.
The generating function for all the nodes of $\T_{132}$ is then
$$R(z)=1+A(1,z)+\wt{A}(1,z)+\wh{A}(1,z)=\frac{1-z-z^2-\sqrt{1-2z-z^2-2z^3+z^4}}{2z^3}.$$
\end{proof}

\begin{proof}[Second proof (permutation decomposition)]
Let $\pi\in\RS_n(132)$ with $n\ge1$. We can write $\pi=\sigma n\tau$, where $\sigma\gg\tau$.
If $\tau$ is empty, then $\sigma$ is an arbitrary permutation in $\RS_{n-1}(132)$.

Assume now that $\tau$ is nonempty.
Then either $\sigma$ has length $0$ or $1$, or the smallest entry in $\sigma$ appears immediately to the left of its second smallest entry. Indeed,
the fact that it appears to the left follows from the simsun condition; otherwise, by deleting all entries larger than the two smallest entries of $\sigma$, these two and the first entry of $\tau$ would form a double descent.
The fact that there are no other entries in between follows from the $132$-avoiding condition. Deleting the
smallest entry of $\sigma$ and reducing its remaining entries we obtain an arbitrary $132$-avoiding simsun permutation.

Additionally, either $\tau$ has length 1 or $\tau_2=\tau_1+1$. Indeed, if $\tau_1>\tau_2$ then $\pi$ would have a double descent $n\tau_1\tau_2$, and if $\tau_2>\tau_1+1$ then $\tau$ would contain $132$.
Deleting $\tau_1$ and reducing the remaining entries of $\tau$ we obtain again an arbitrary $132$-avoiding simsun permutation.
Finally, note that these conditions in $\sigma$ and $\tau$ are sufficient to guarantee that $\pi=\sigma n\tau\in\RS_n(132)$.

It follows from this decomposition that the generating function $R(z)$ satisfies the equation
\beq\label{eq:R}R(z)=1+zR(z)+z(1+zR(z))zR(z)=1+zR(z)+z^2R(z)+z^3R(z)^2.\eeq
The summand $zR(z)$ corresponds to the case when $\tau$ is empty. The factor $(1+zR(z))$ records that $\red(\sigma)$ is either empty or can be obtained from a $132$-avoiding simsun permutation by shifting all the entries up by one and inserting 1 right before
the smallest entry. The term $zR(z)$ at the end comes from the fact that when $\tau$ is nonempty, it can be obtained from a $132$-avoiding simsun permutation by inserting a copy of the first entry at the beginning of the permutation and increasing all
the entries greater than or equal to it by one.

Solving equation~(\ref{eq:R}) we obtain the generating function for the secondary structure numbers.
\end{proof}

From the above proof one can derive a recursive bijection $\bij$ from $\RS_n(132)$ to the set $\D_n^1$ of Dyck paths of semilength $n$ with no $UUU$ and no $DDD$,
which has cardinality $S_n$ (see Subsection~\ref{sec:Dyck}). Let $\red$ denote the usual reduction, and let $\red_1(\sigma)$ denote the permutation
obtained by removing the smallest entry from $\sigma$ and reducing the rest. For $\pi\in\RS_n(132)$, define $\bij(\pi)$ as follows:
\bit\item if $\pi$ is empty, let $\bij(\pi)$ be the empty path;
\item if $\pi=\sigma n$, let $\bij(\pi)=UD\bij(\sigma)$;
\item if $\pi=n \tau_1\dots\tau_{n-1}$ with $n\ge2$, let $\bij(\pi)=UUDD\bij(\red(\tau_2\dots\tau_{n-1}))$;
\item if $\pi=\sigma n\tau$ with $\sigma$ and $\tau$ nonempty, let $\bij(\pi)=UUD\bij(\red_1(\sigma))UDD\bij(\red(\tau_2\tau_3\dots))$.
\eit

As an example of the decomposition in the second proof above,
consider $\pi=75682341\in\RS_8(132)$. Then $\sigma=756$,
from where deleting the smallest entry and reducing we obtain $21\in\RS(132)$, and $\tau=2341$, from
where deleting the first entry and reducing we obtain $231\in\RS(132)$. From $21$ and $231$ we can reconstruct $\pi$ as described in the proof.
Applying $\bij$ to this example we get
$$\bij(75682341)=UUD\bij(21)UDD\bij(231)=UUDUUDD\,UDDUUD\,UDD,$$
which is drawn in Figure~\ref{fig:bij132}.

\begin{figure}[hbt]
\centering
\epsfig{file=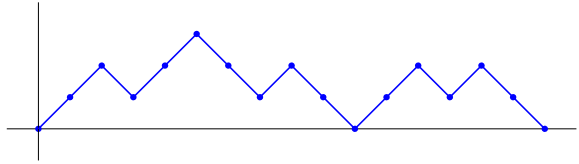,width=3in}
\caption{\label{fig:bij132} The Dyck path $\bij(75682341)$, which has no $UUU$ and no $DDD$.}
\end{figure}

It follows easily by induction that
the number of descents of $\pi$ equals the number of occurrences of $UU$ in $\bij(\pi)$. Indeed, in the
last two bullets of the above recursive definition of $\bij$, a new descent is added to the permutation and a new
$UU$ is added to the path.

\begin{proof}[Third proof (bijection)]
Here we give a nonrecursive bijection between $\RS_n(132)$ and Dyck paths of semilength $n+2$ with no $UDU$ and no $DUD$,
which are also counted by $S_n$ (see Subsection~\ref{sec:Dyck}).
We claim that when $\krar_{132}$, defined in Section~\ref{sec:krar}, is restricted to simsun permutations, it induces a bijection between $\RS_n(132)$ and the set $\D'_n$ of
Dyck paths of semilength $n$ with the property that all the peaks, with the possible exception of the first and the last one, are of the form $UUDD$.
In other words, $\D'_n$ consists of paths having no ascents of length 1 and no descents of length 1 between consecutive valleys.
Our claim follows easily from Lemma~\ref{lem:132lrm} below: (a) is equivalent to $\krar_{132}(\pi)$ having no descents of length 1 between consecutive valleys,
and (b) is equivalent to $\krar_{132}(\pi)$ having no ascent of length 1 between consecutive valleys.

The set $\D'_n$ is in bijection with the set $\D^2_{n+2}$ of Dyck paths of semilength $n+2$ having all peaks of the form $UUDD$, or equivalently,
having no $UDU$ and no $DUD$. The bijection is very simple: given a path in $\D'_n$, add a $U$ to the first and last ascents, and a $D$ to the first and last descents
to obtain a path in $\D^2_{n+2}$ (if the path has only one peak, then add two $U$ steps to its only ascent and two $D$ steps to its only descent). See Figure~\ref{fig:bij132nonr} for an example.
\end{proof}

\begin{figure}[hbt]
\centering
\epsfig{file=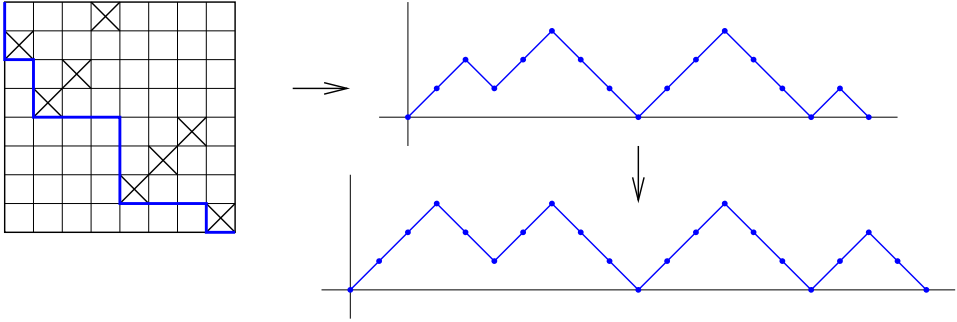,width=5in}
\caption{\label{fig:bij132nonr} The Dyck path with no $UDU$ and no $DUD$ corresponding to $75682341$.}
\end{figure}

\begin{lemma}\label{lem:132lrm}
Let $\pi\in\S_n(132)$, and let $l_1>l_2>\dots>l_r$ be its left-to-right minima. Then $\pi$ is simsun if and only if
\ben\renewcommand{\labelenumi}{(\alph{enumi})}
\item $\pi$ has no double descent, and
\item $l_{j-1} - l_{j} \ge2$ for all $2\le j\le r-1$.
\een
\end{lemma}

\begin{proof}
Assume that $\pi$ is simsun. Then, it clearly has no double descent. Suppose now that (b) does not hold, so there is some $2\le j\le r-1$ such that
$l_{j-1}=l_{j}+1$. Then we claim that $l_{j-1}l_j l_{j+1}$ is a double descent in the permutation $\pi$ restricted to $\{1,2,\dots,l_{j-1}\}$. Indeed, since $l_j$ and $l_{j+1}$
are consecutive left-to-right minima, all the entries to the left of $l_{j+1}$ are at least as large as $l_j$, so with the exception of $l_{j-1}$ and $l_j$ they are all
greater than $l_{j-1}$.

Assume now that $\pi$ satisfies (a) and (b), and suppose it is not simsun. Then there exist indices $a<b<c$ such that $\pi_a>\pi_b>\pi_c$ and $\pi_i>\pi_a$
for all $a<i<b$ and all $b<i<c$. Additionally, $\pi_i>\pi_b$ for all $i<a$, otherwise $\pi_i\pi_a\pi_b$ would be an occurrence of $132$. Thus, $\pi_b$ is a
left-to-right minimum, say $\pi_b=l_j$, where $2\le j\le r-1$. If $\pi_b+1$ was to the left of $\pi_b$ (in the one-line notation of $\pi$), then it would also be
a left-to-right minimum, so we would have $l_{j-1}-l_j=(\pi_b+1)-\pi_b=1$, contradicting (b). Hence $\pi_b+1$ is to the right of $\pi_b$, and in fact to the right
of $\pi_c$ because we have seen that $\pi_i>\pi_a$ for all $b<i<c$.
Now we look at $\pi_{b+1}$. If $\pi_{b+1}<\pi_b$, then $\pi_{b-1}\pi_b\pi_{b+1}$ would be a double descent in $\pi$, contradicting (a). But if $\pi_{b+1}>\pi_b$,
then $\pi_b\pi_{b+1}\pi_b+1$ is an occurrence of $132$ in $\pi$, which is a contradiction as well.
\end{proof}

\subsection{Restricted Dyck paths}\label{sec:Dyck}

We have defined $\D^1_n$ (resp. $\D^2_n$) to be the set of Dyck paths of semilength $n$ with no $UUU$ and no $DDD$ (resp. having all peaks of the form $UUDD$). For $n\ge2$,
the restriction for a path to be in $\D_n^2$ is equivalent to the path having no $UDU$ and no $DUD$.
 The fact that $|\D^1_n|=|\D^2_{n+2}|=S_n$ is mentioned in~\cite{Don}.
Here we provide easy proofs using generating functions.

Every path $P\in\D^1_n$ with $n\ge1$ can be written as either $P=UDQ$, $P=UUDDQ$, or $P=UUDQUDDQ'$, where $Q$ and $Q'$ are arbitrary Dyck paths with no $UUU$ and no $DDD$.
It follows that the generating function $F(z)=\sum_{n\ge0}|\D^1_n|z^n$ satisfies $$F(z)=1+zF(z)+z^2F(z)+z^3F(z)^2,$$ from where
$$F(z)=\frac{1-z-z^2-\sqrt{1-2z-z^2-2z^3+z^4}}{2z^3},$$
which is the generating function for the sequence $S_n$ (see Table~\ref{tab:seq}).

Similarly, every path $P\in\D^2_n$ with $n\ge1$ can be written as either $P=UQDQ'$ or $P=UUDDQ'$, where  $Q,Q'$
are Dyck paths having all peaks of the form $UUDD$ and, moreover, $Q$ is nonempty. It follows that the generating function $G(z)=\sum_{n\ge0}|\D^2_n|z^n$ satisfies $$G(z)=1+z(G(z)-1+z)G(z),$$ from where
$$G(z)=\frac{1+z-z^2-\sqrt{1-2z-z^2-2z^3+z^4}}{2z}=1+z^2F(z).$$

\section{Avoiding $213$}

Simsun permutations avoiding $213$ are counted by the Motzkin numbers. We give three proofs.

\begin{prop}\label{prop:213}
$$|\RS_n(213)|=M_n.$$
\end{prop}

\begin{proof}[First proof (generating tree)]
Let $\pi\in\RS_n(213)$, and let $j$ be such that $\pi_j=n$. Since $\pi$ is $213$-avoiding, it has no descents to the left of $\pi_j$. If $j=n$, then $\pi=12\dots n$; otherwise $\pi_1<\pi_2<\dots<\pi_{j-1}<n>\pi_{j+1}$.
Consider the subtree $\T_{213}$ of $\T$ consisting of $213$-avoiding simsun permutations.
In each of the above two cases, the fertility positions of $\pi$ where the insertion
of $n+1$ produces a permutation in $\RS_{n+1}(231)$ are the ones marked by the upper dots: $\udt1\udt2\udt3\udt\dots\udt n\udt$ ($n+1$ fertility positions), $\udt\pi_1\udt\pi_2\udt\dots\udt\pi_{j-2}\udt\pi_{j-1}n\udt\pi_{j+1}\dots\pi_n$
($j$ fertility positions). It follows that a permutation $\pi\in\RS_n(213)$ with $k$ fertility positions generates $k$ permutations with $1,2,\dots,k-1,k+1$ fertility positions, respectively.

If we label each $\pi\in\RS(213)$ with its number $k$ of fertility positions, we obtain the succession rule
\beq\label{eq:suc213}(k)\longrightarrow(1)(2)\dots(k-1)(k+1)\eeq
for the labels of $\T_{213}$, with the root having label $(1)$, corresponding to the empty permutation. Denoting by $B(u,z)$ the generating function where the coefficient of $u^kz^n$ is the number of
nodes at level $n$ with label $(k)$, the above succession rule translates into the functional equation
$$B(u,z)=u+z\left(\frac{B(u,z)-uB(1,z)}{u-1}+uB(u,z)\right).$$
Collecting the terms with $B(u,z)$ and canceling the kernel, we obtain
$$B(1,z)=\frac{1-z-\sqrt{1-2z-3z^2}}{2z^2},$$
which is the generating function for the Motzkin numbers.

Alternatively, to the succession rule~(\ref{eq:suc213}) there corresponds the production matrix~\cite{DFR1}
$$\left(
  \begin{array}{ccccccc}
    0 & 1 & 0 & 0 & 0 & 0 & \dots \\
    1 & 0 & 1 & 0 & 0 & 0 & \dots \\
    1 & 1 & 0 & 1 & 0 & 0 & \dots \\
    1 & 1 & 1 & 0 & 1 & 0 & \dots \\
    1 & 1 & 1 & 1 & 0 & 1 & \dots \\
    \vdots & \vdots & \vdots & \vdots & \vdots & \vdots & \ddots \\
  \end{array}
\right).$$
Then, taking $b=0$, $c=1$, and $r=1$ in Corollary 3.1 of~\cite{DFR1}, we obtain that the induced generating function $f=f(z)$ satisfies $f = 1 + zf +z^2f^2$. One can also use Propositions 3.2 and 3.3 of~\cite{DFR2}.
\end{proof}

\begin{proof}[Second proof (bijection)]
We describe a bijection from Motzkin paths to $213$-avoiding simsun permutations.
It is in fact a composition of two bijections, one from $\M_n$ to $\A_n$ and one from $\A_n$ to $\RS_n(213)$,
where
\beq\label{eq:An}\A_n=\{(a_1,a_2,\dots,a_n):a_1=1 \mbox{ and for all }i\ge1,\, a_{i+1}=a_i+1 \mbox{ or } 0\le a_{i+1}< a_i\}.\eeq
The bijection between $\M_n$ and $\A_n$ is presented in Section~\ref{sec:Motzkin}.

We now describe the bijection from $\A_n$ to $\RS_n(213)$. Given $(a_1,a_2,\dots,a_n)\in\A_n$,
start from the empty permutation and, for $i$ from $1$ to $n$, insert $i$ so that it leaves $a_i$ of the previously inserted entries
to its left, unless $a_i$ is larger than the number $i-1$ of entries so far, in which case we insert $i$ in the rightmost position.
We claim that the resulting permutation is in $\RS_n(213)$, and that the map is a bijection.

Indeed, this map amounts to interpreting the sequence $(a_1+1,a_2+1,\dots,a_n+1)$ as the labels of a path starting at the root of $\T_{213}$.
The succession rule~(\ref{eq:suc213}) guarantees that the sequence describes a path, and that every path of length $n$ determines a sequence in $\A_n$.
The map associates each sequence to the permutation at the end of the path.

For example, the permutation corresponding to the sequence $$(1,2,0,1,2,3,4,5,4,2,3,4,3,4,1)\in\A_{15},$$ which in turn comes from the path in~Figure~\ref{fig:motzkin213}, is
$$3\, 15\, 4\, 10\, 13\, 14\, 11\, 12\, 5\, 6\, 9\, 7\, 8\, 1\, 2.$$
\end{proof}

\begin{proof}[Third proof (bijection)]
The bijection $\krar_{213}$ shown in Figure~\ref{fig:krar} can be accurately described as follows.
Given the array of a permutation $\pi\in\S_n(213)$, consider the path from the upper-left corner to the lower-right corner
with steps east and south that leaves all the crosses to its left and stays always as close to the diagonal connecting these two corners as possible.
Now replace each east step with a $U$ and each south step with a $D$ to obtain a Dyck path $\krar_{213}(\pi)\in\D_n$.

We claim that $\pi$ is simsun if and only if the path $\krar_{213}(\pi)$ does not contain a $DUD$ except possibly at the last peak.
To see this, let $\pi\in\S_n(213)$. Assume first that is not simsun, so there exist indices $a<b<c$ with $\pi_a>\pi_b>\pi_c$ and so that
$\pi_i>\pi_a$ for all $a<i<b$ and all $b<i<c$. Since $\pi$ is $213$-avoiding, $\pi_b$ must be a right-to-left maximum, and $\pi_{b-1}\ge\pi_a>\pi_b$.
It follows that the step in $\krar_{213}(\pi)$ preceding the peak caused by the cross $(b,\pi_b)$ is a $D$, so the path contains a $DUD$, and this is not the last
peak of the path because $b<n$.

Conversely, if $\krar_{213}(\pi)$ contains a $DUD$ before the last peak, let $(b,\pi_b)$ be the cross creating the peak inside this $DUD$. Then $\pi_b$ is a right-to-left maximum, so $\pi_{b+1}<\pi_b$,
and also $\pi_{b-1}>\pi_b$, otherwise the peak would not be preceded by a $D$. But this means that $\pi_{b-1}\pi_b\pi_{b+1}$ is a double descent, so $\pi$ is not simsun.

From a Dyck path of semilength $n$ with no $DUD$ except possibly at the last peak we can obtain a Dyck path of semilength $n+1$ with no $DUD$ by adding a $U$ to the last ascent and a $D$ to the last descent of the path.
This process is clearly reversible. Finally, we apply a bijection due to Callan~\cite{Cal} between Dyck paths of semilength $n+1$ with no $DUD$ and Motzkin paths of length $n$:
mark each $D$ that is preceded by a $U$ and each $U$ that is not preceded by a $D$, change each unmarked $D$ whose matching $U$ is marked to an $H$, and delete all marked steps.

Figure~\ref{fig:motzkin213cal} shows an example of the stages of this bijection applied to $$\pi=1\,13\,2\,8\,11\,12\,9\,10\,3\,4\,5\,7\,6\,\in\RS_{13}(213).$$

\begin{figure}[hbt]
\centering
\epsfig{file=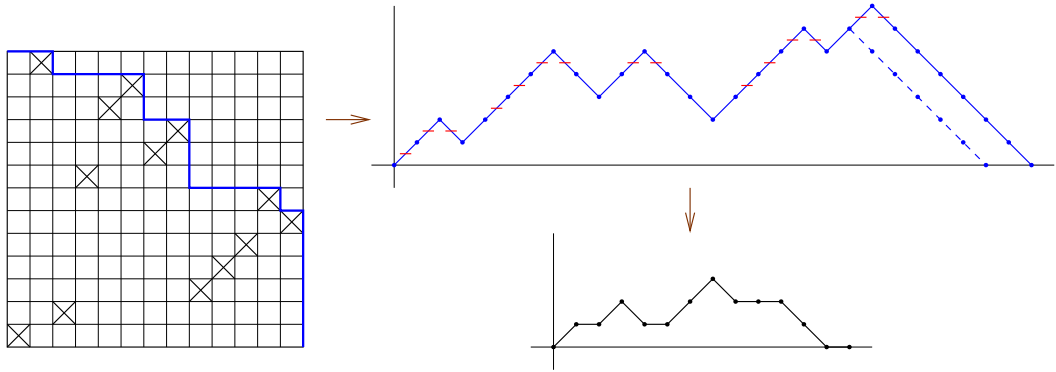,width=5.5in}
\caption{\label{fig:motzkin213cal} Starting from $\pi=1\,13\, 2\, 8\, 11\, 12\, 9\, 10\, 3\, 4\, 5\, 7\, 6\in\RS_{13}(213)$, we first get a Dyck path with no $DUD$ and then a Motzkin path.}
\end{figure}
\end{proof}

Simsun permutations avoiding $213$ can be characterized in terms of generalized pattern avoidance.

\begin{prop}
$$\RS_n(213)=\S_n(213,\wh{321}).$$
\end{prop}

\begin{proof}
The inclusion $\RS_n(213)\subseteq\S_n(213,\wh{321})$ follows from~(\ref{eq:inclRS}).
Now assume that $\pi\in\S(213)$ is not simsun, and let $a<b<c$ be such that $\pi_a>\pi_b>\pi_c$ and
$\pi_i>\pi_a$ for all $a<i<b$ and all $b<i<c$.
If there was an entry $\pi_i$ with $b<i<c$, then $\pi_a\pi_b\pi_i$ would be an occurrence of $213$. It follows that $c=b+1$. Now, since $\pi_{b-1}>\pi_b$, we have that $\pi_{b-1}\pi_b\pi_{b+1}$ is an occurrence of $\wh{321}$.
\end{proof}

\subsection{A bijection for Motzkin paths}\label{sec:Motzkin}

Here we describe a bijection between $\M_n$ and the set $\A_n$ defined in equation~(\ref{eq:An}).

Given $M\in\M_n$, we first label the $n+1$ vertices of the path, that is, the lattice points between consecutive steps of the path, plus the initial and final vertices $(0,0)$ and $(n,0)$.
The label of each vertex $v$ is the number of steps that $M$ has in common with the path that
starts at $v$ and ends at $(n,0)$, has only $H$ and $D$ steps, never goes above $M$, and always stays as close to $M$ as possible.
Ignoring the rightmost label, which is always $0$, the remaining labels from right to left form a sequence $(a_1,a_2,\dots,a_n)\in\A_n$.

To see that this sequence is in $\A_n$, note that $a_{i+1}=a_i+1$ if the $(n-i)$-th step of $M$ is a $D$ or an $H$, and $0\le a_{i+1}<a_i$
if the step is a $U$. To recover $M$ from a given $(a_1,a_2,\dots,a_n)\in\A_n$, draw the steps from left to right as follows. For $j$ from $1$ to $n$, draw a $U$ if $a_{n-j+1}<a_{n-j}$. Otherwise, if $b$ is the $y$-coordinate
of the vertex we are at and $b>0$, let $a_k$ be the label of the last vertex with $y$-coordinate $b-1$; draw a $D$ if $a_{n-j}=a_k$, and an $H$ if $a_{n-j}>a_k$ or $b=0$.

For example, for the path $M=UHUDHUUDHDDHUHD\in\M_{15}$, the labels are shown in Figure~\ref{fig:motzkin213}.
\begin{figure}[hbt]
\centering
\epsfig{file=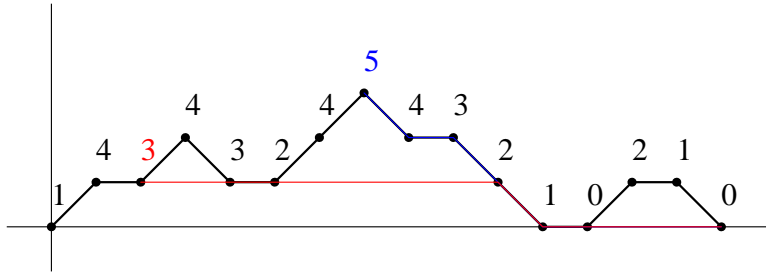,width=4in}
\caption{\label{fig:motzkin213} The bijection from Motkzin paths to sequences in $\A_n$.}
\end{figure}

To the best of our knowledge, this construction gives a new manifestation of the Motzkin numbers as counting sequences in $\A_n$.

Suppose now that the path $M\in\M_n$ is mapped by this bijection to the sequence $(a_1,\dots,a_n)\in\A_n$, and that $\pi\in\RS_n(213)$ is the corresponding permutation following the second proof of Proposition~\ref{prop:213}.
From the definition of $a_n$ it is clear that it equals the number of $H$ steps on the $x$-axis in $M$.
Additionally, the number of $U$ steps in $M$ equals the number of indices with $a_i>a_{i+1}$,
which in turn equals the number of descents of $\pi$ and the number of descents of $\pi^{-1}$. This follows from the fact that when inserting $n+1$ in $\pi\in\RS_n(213)$, the number of descents
of $\pi$ (resp. $\pi^{-1}$) increases by one unless $n+1$ is inserted in the rightmost fertility position, in which case it stays the same.

\section{Avoiding $231$}

Simsun permutations avoiding $231$ are also counted by the Motzkin numbers, although the proofs are different than in the previous section.

\begin{prop}
$$|\RS_n(231)|=M_n.$$
\end{prop}

\begin{proof}[First proof (permutation decomposition)]
A permutation $\pi\in\RS_n(231)$ with $n\ge1$ can be decomposed as $\pi=\sigma n\tau$, where $\sigma\ll\tau$, $\sigma\in\RS(231)$, and $\red(\tau)\in\RS^\uparrow(231)$, which
is the set of permutations in $\RS(231)$ that do not start with a descent. Note that these conditions on $\sigma$ and $\tau$ are sufficient to guarantee that $\pi$ is simsun.
Similarly, any $\pi\in\RS^\uparrow_n(231)$ with $n\ge1$ is of the form $1\tau$, where $\rho(\tau)\in\RS_{n-1}(231)$.
Letting $H(z)=\sum_{n\ge0}|\RS_n(231)|z^n$ and $H^\uparrow(z)=\sum_{n\ge0}|\RS^\uparrow_n(231)|z^n$, the above decompositions translate into the equations
$$\left\{\begin{array}{l} H(z)=1+z H(z) H^\uparrow(z), \\ H^\uparrow(z)= 1 + zH(z). \end{array}\right.$$
Eliminating $H^\uparrow(z)$ we obtain $H(z)= 1 + zH(z) + z^2 H(z)^2$, and thus
$$H(z)=\frac{1-z-\sqrt{1-2z-3z^2}}{2z^2}.$$

\end{proof}

\begin{proof}[Second proof (generating tree)]
Let $\T_{231}$ be the subtree of $\T$ consisting of $231$-avoiding simsun permutations. The fertility positions of a permutation in this tree are
those spots that do not precede a descent and such that all the entries to their left are smaller than all the entries to their right. Note that this always includes the rightmost position of the permutation, which we call the
terminal fertility position. If $n+1$ is inserted in a non-terminal fertility position of $\pi\in\RS_n(231)$, then the fertility positions of the resulting permutation are the fertility positions of $\pi$ to
the left of where the insertion took place. If $n+1$ is inserted in the terminal position, all the fertility positions are preserved and a new one is created. It follows that a permutation with $k$ fertility
positions generates $k$ permutations with $1,2,\dots,k-1,k+1$ fertility positions. Now we continue as in the first proof of Proposition~\ref{prop:213}.
\end{proof}

\begin{proof}[Third proof (bijection)]
By Proposition~\ref{prop:231gen} below, $\RS_n(231)=\S_n(231,\wh{32}1)$. Permutations in this set are reversals of permutations in $\S_n(132,1\wh{23})$, that is,
$\pi_1\dots\pi_n\in\S_n(231,\wh{32}1)$ if and only if $\pi_n\dots\pi_1\in\S_n(132,1\wh{23})$. The latter set was studied in~\cite{EliMan}, where a
bijection between $\S_n(132,1\wh{23})$ and $\M_n$ is given. The resulting bijection $\bijr$ from $\RS_n(231)$ to $\M_n$ can be described as follows: apply $\krar_{231}$ to $\pi\in\RS_n(231)$ to obtain a UUU-free Dyck path, then
divide the Dyck path into pieces of the form $D$, $UD$ and $UUD$, and replace the pieces of each kind with $D$, $H$ and $U$, respectively (see Figure~\ref{fig:bijr}).

Alternatively, we can give the following recursive description of $\bijr$.
Let the image of the empty permutation be the empty path, and let $\bijr(1)=H$.
If $\pi\in\RS_n(231)$ is decomposable, that is, $\pi=\sigma\tau$ where $\sigma\ll\tau$ and $\sigma$ and $\tau$ are nonempty, let $\bijr(\pi)$ be the concatenation of the paths $\bijr(\sigma)$ and $\bijr(\red(\tau))$.
If $\pi=n1\sigma$, let $\bijr(\pi)=U\bijr(\red(\sigma))D$. We claim that this covers all the cases. Indeed, if $\pi$ is indecomposable then $\pi_1=n$, otherwise  avoidance of $231$ would force all the entries
to the left of $n$ to be smaller than all the entries to its right, contradicting indecomposability. Now $\pi_2=1$, otherwise $\pi_1\pi_2\pi_3$ would be a double descent (if $\pi_2>\pi_3$) or $\pi_2\pi_31$ would be an occurrence of $231$ (if $\pi_2<\pi_3$).
\end{proof}

\begin{figure}[hbt]
\centering
\epsfig{file=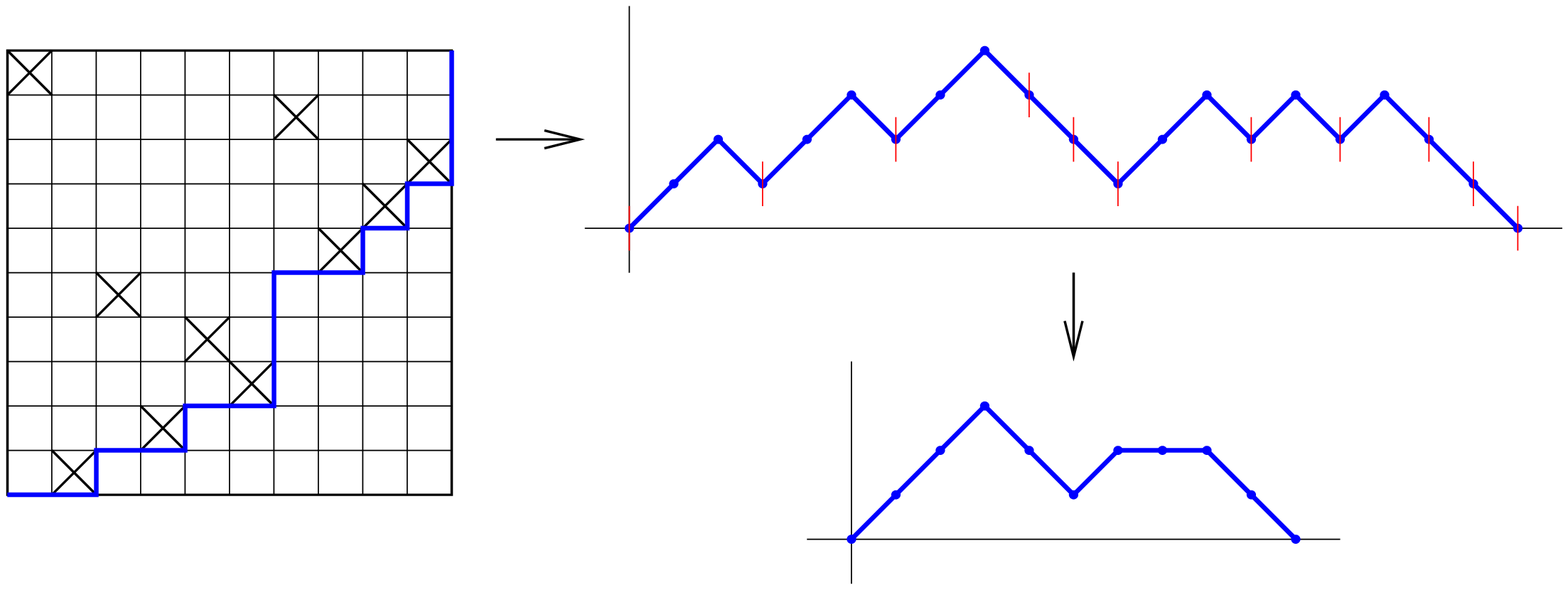,width=5in}
\caption{\label{fig:bijr} The bijection $\bijr$ applied to $10\, 1\, 5\, 2\, 4\, 3\, 9\, 6\, 7\, 8$.}
\end{figure}

The recursive description of $\bijr$ in the previous proof reveals the following statistic correspondences, which are easy to prove by induction:
\bit
\item the number of inversions of $\pi$ equals the area under $\bijr(\pi)$,
\item the number of descents of $\pi$ equals the number of $U$ steps of $\bijr(\pi)$,
\item the number of right-to-left maxima of $\pi$ equals one plus the length of the final descent of $\bijr(\pi)$.
\eit

Now we prove the characterization of simsun permutations avoiding $231$ in terms of generalized pattern avoidance.

\begin{prop}\label{prop:231gen}
$$\RS_n(231)=\S_n(231,\wh{321})=\S_n(231,\wh{32}1).$$
\end{prop}

\begin{proof}
By the inclusions~(\ref{eq:inclRS}) restricted to $231$-avoiding permutations we have that $$\S_n(231,\wh{32}1)\subseteq\RS_n(231)\subseteq\S_n(231,\wh{321}),$$ so
it suffices to show that $\S_n(231,\wh{321})\subseteq\S_n(231,\wh{32}1)$.
Assume that $\pi\in\S_n(231)$ contains an occurrence $\pi_i\pi_{i+1}\pi_j$ of $\wh{32}1$. We must have $\pi_{i+2}<\pi_{i+1}$, otherwise $\pi_{i+1}\pi_{i+2}\pi_j$ would be an occurrence of $231$.
But then $\pi_i\pi_{i+1}\pi_{i+2}$ is an occurrence of $\wh{321}$.
\end{proof}

\section{Avoiding $312$}

The enumeration of simsun permutations avoiding $312$ is relatively straightforward.

\begin{prop}
$$|\RS_n(231)|=2^{n-1}.$$
\end{prop}

\begin{proof}[First proof (generating tree)]
Consider the subtree of $\T$ consisting of $312$-avoiding simsun permutations. For $\pi\in\RS_n(312)$, the fertility positions where inserting $n+1$ produces a permutation in $\RS_{n+1}(312)$ are
those that do not precede a descent (because of the simsun condition) and do not have any ascent to their right (because of $312$ avoidance). The only two such positions are the last two, so the number of nodes
in the tree double at each level.
\end{proof}

\begin{proof}[Second proof (bijection)]
By Proposition~\ref{prop:312gen} below, $\RS_n(312)=\S_n(312,321)$. These permutations were enumerated in~\cite{SS}.
\end{proof}

\begin{prop}\label{prop:312gen}
$$\RS_n(312)=\S_n(312,321)=\S_n(312,\wh{32}1).$$
\end{prop}

\begin{proof}
By the inclusions~(\ref{eq:inclRS}) and the fact that an occurrence of $\wh{32}1$ is also an occurrence of $321$, we have that $$\S_n(312,321)\subseteq\S_n(312,\wh{32}1)\subseteq\RS_n(312),$$ so
it suffices to show that $\RS_n(312)\subseteq\S_n(312,321)$. Assume that $\pi\in\S_n(312)$ contains an occurrence $\pi_a\pi_{b}\pi_c$ of $321$.
We show that $\pi$ is not simsun by producing an occurrence $\pi_{b-1} \pi_b \pi_d$ of $321$ such that all entries
between $\pi_b$ and $\pi_d$ are larger than $\pi_{b-1}$.
First we see that $\pi_{b-1}>\pi_b$, otherwise $a<b-1$ and $\pi_a\pi_{b-1}\pi_b$
would be an occurrence of $312$. Let $d$ be the smallest index such that $d>b$ and $\pi_d<\pi_b$. Clearly $d\le c$. Then $\pi_{b-1}>\pi_b>\pi_d$, and for all $b<i<d$, $\pi_i>\pi_{b-1}$ (otherwise $\pi_{b-1}\pi_b\pi_i$ would be
an occurrence of $312$), so $\pi$ is not simsun.
\end{proof}

\section{Double restrictions}\label{sec:double}

In this section we consider simsun permutations avoiding simultaneously two patterns of length~3.

\begin{lemma} For $n\ge4$,
$$|\RS_n(123,132)|=2.$$
\end{lemma}

\begin{proof} The first four levels of the subtree of $\T_{123}$ containing only $132$-avoiding permutations are drawn in Figure~\ref{fig:RS123132} (left). Starting at the fourth level, each permutation has one fertility position
right before the first ascent pair, as in the proof of Proposition~\ref{prop:RS123}. This is either the leftmost position or the position right after $\pi_1=n$, so insertion of $n+1$ in this positions
never creates an occurrence of $132$.
\end{proof}

\begin{figure}[htb]
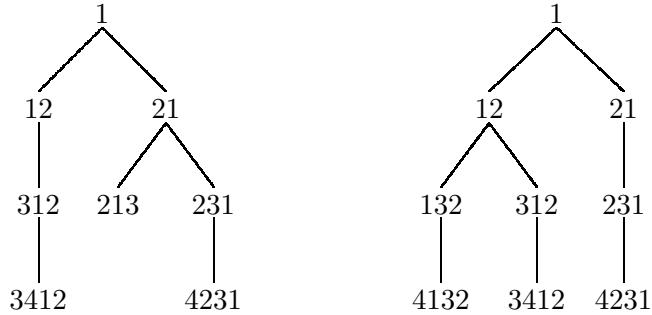

\begin{center}\synttree
[$1$ [$12$ [$312$ [$3412$]]] [$21$ [$213$] [$231$ [$4231$]]]] \hspace{2cm}
\synttree
[$1$ [$12$ [ $132$ [$4132$]] [$312$ [$3412$] ]] [$21$ [$231$ [$4231$]]]]
\end{center}
\caption{The first four levels of the generating trees for $\RS(123,132)$ (left) and $\RS(123,213)$ (right). \label{fig:RS123132}}
\end{figure}

\begin{lemma} For $n\ge3$,
$$|\RS_n(123,213)|=3.$$
\end{lemma}

\begin{proof} The first four levels of the subtree of $\T_{123}$ containing only $213$-avoiding permutations are drawn in Figure~\ref{fig:RS123132} (right). Starting at the fourth level, each permutation has one fertility position
right before the first ascent pair. This is always one of the two leftmost positions, so the insertion never creates an occurrence of $213$.
\end{proof}

\begin{lemma} For $n\ge5$,
$$|\RS_n(123,231)|=|\RS_n(123,312)|=0.$$
\end{lemma}

\begin{proof} The subtrees of $\T_{123}$ containing only $231$-avoiding permutations and only $312$-avoiding permutations, respectively, are drawn in Figure~\ref{fig:RS123231}.
The permutations of length 4 have no fertility positions in either tree.
\end{proof}

\begin{figure}[htb]
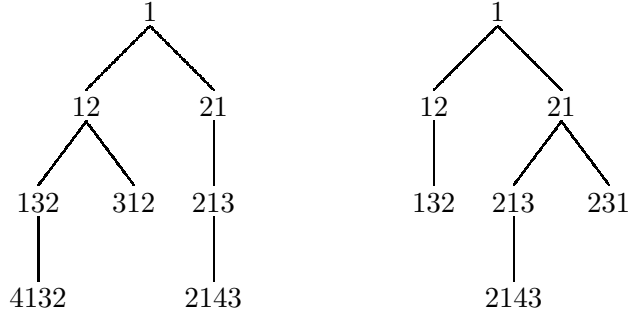

\begin{center}\synttree
[$1$ [$12$ [ $132$ [$4132$]] [$312$]] [$21$ [$213$ [$2143$]]]] \hspace{2cm}
\synttree
[$1$ [$12$ [$132$]] [$21$ [$213$ [$2143$]] [$231$]]]
\end{center}
\caption{The generating trees for $\RS(123,231)$ (left) and $\RS(123,312)$ (right). \label{fig:RS123231}}
\end{figure}

\begin{lemma}
$$|\RS_n(132,213)|=F_{n+1}.$$
\end{lemma}

\begin{proof}
Consider the subtree of $\T_{213}$ containing only $132$-avoiding permutations.
In the first proof of Proposition~\ref{prop:213} we saw that the first ascending run of $\pi\in\RS_n(213)$ ends in $n$. Among the fertility positions mentioned in that proof,
only the leftmost position (provided $\pi_1\neq n$ or $n=1$) and the position immediately to the right of $n$ produce a $132$-avoiding permutation when $n+1$ is inserted.
It follows that for $n\ge2$, in this subtree $\pi\in\RS_n(213)$ has one fertility position if $\pi_1=n$ and two fertility positions otherwise. Labeling the nodes with the number of fertility positions,
the succession rule is
\bea\nn&&(1)\longrightarrow(2),\\
\nn&&(2)\longrightarrow(1)(2),\eea
with root $(1)$ corresponding to the empty permutation. It is well-known (see, for example~\cite{W}) that the number of nodes at each level of a tree that follows this succession rule is given by the Fibonacci numbers.
\end{proof}

\begin{lemma}
$$|\RS_n(132,231)|=n.$$
\end{lemma}

\begin{proof}
If $\pi\in\RS_n(132,231)$, then the only possible $i$ with $\pi_i>\pi_{i+1}$ is $i=1$. Otherwise, $\pi_{i-1}\pi_i\pi_{i+1}$ would be a double descent or an occurrence of $231$ or $132$.
On the other hand, it is clear that the $n$ permutations satisfying $\pi_2<\pi_3<\dots<\pi_n$ are in $\RS_n(132,231)$.
\end{proof}

\begin{lemma}
$$|\RS_n(213,231)|=F_{n+1}.$$
\end{lemma}

\begin{proof}
We show that any $\pi\in\RS_n(213,231)$ with $n\ge1$ can be written as $\pi=1\sigma$ or as $\pi=n1\tau$, where $\rho(\sigma)$ and $\rho(\tau)$ are arbitrary permutations in $\RS_{n-1}(213,231)$ and
$\RS_{n-2}(213,231)$ respectively.
To see this, note that if $1<\pi_1<n$, then $\pi_1$ would be the first element of an occurrence of $213$ or $231$ in $\pi$.
If $\pi_1=n\ge3$, then $\pi_2\pi_3\dots\pi_n\in\RS_{n-1}(213,231)$, and applying the same argument we see that $\pi_2=1$ or $\pi_2=n-1$, but in the second case $\pi_1\pi_2\pi_3$ would be a double descent.
Finally, it is clear that given any $\sigma\in\RS_{n-1}(213,231)$, adding one to its entries and inserting a 1 at the beginning produces a permutation in $\RS_n(213,231)$, and inserting $n+1$ at the beginning
of the resulting permutation produces an element of $\RS_{n+1}(213,231)$.
The result follows from this recursive construction and the fact that $|\RS_0(213,231)|=|\RS_1(213,231)|=1$.
\end{proof}

\begin{lemma}
\bea\nn&|\RS_n(132,312)|=|\RS_n(213,312)|=n,&\\
\nn&|\RS_n(231,312)|=F_{n+1}.&
\eea
\end{lemma}

\begin{proof}
By Proposition~\ref{prop:312gen}, we have that
\bea\nn &&\RS_n(132,312)=\S_n(132,312,321),\\
\nn &&\RS_n(213,312)=\S_n(213,312,321),\\
\nn &&\RS_n(231,312)=\S_n(231,312,321).
\eea
These sets have been enumerated in~\cite{SS}.
\end{proof}

In all the cases of the form $\RS_n(\sigma,321)$ where $\sigma$ is a pattern of length 3, the condition of avoiding $321$ overrides the simsun restrictions. Thus, $\RS_n(\sigma,321)=\S_n(\sigma,321)$, and
one can refer to~\cite{SS}.

We end with a curious fact that can be obtained from Theorem~\ref{th:rs} and Tables~\ref{tab:1or2} and~\ref{tab:3up}. Using inclusion-exclusion, we see that for $n\ge5$, the number of simsun permutations in $\S_n$ which contain all six
patterns of length $3$ is
$$E_{n+1}-C_n-2M_n-S_n+2F_{n+1}+2^{n-1}+n^2-3n-1.$$
The first seven of these numbers, corresponding to $5\le n\le 11$, are $1$, $76$, $753$, $5910$, $43985$, $332401$, $2631499$. The term $1$ counts the permutation $41352$.

\end{document}